\documentclass[12pt]{article}
\usepackage{amsmath}
\usepackage{amsfonts}
\usepackage{amssymb}

\newtheorem{teor}{Theorem}[section]
\newtheorem{prop}[teor]{Proposition}
\newtheorem{coro}[teor]{Corollary}
\newtheorem{lema}[teor]{Lemma}
\newtheorem{defi}{Definition}[section]

\newtheorem{nota}{Remark}[section]

\newenvironment{proof}[1][Proof]{\textbf{#1.} }{\ \rule{0.5em}{0.5em}}

\begin{document}
\def\N{\mathbb{N}}
\def\Z{\mathbb{Z}}
\def\K{\mathbb{K}}
\def\R{\mathbb{R}}
\def\D{\mathbb{D}}
\def\C{\mathbb{C}}
\def\T{\mathbb{T}}

\def\dint{\int}
\def\intc{\int_0^1}
\def\intD{\int_{\D}}
\def\intpi{\int_0^{2\pi}}
\def\sumi{\sum_{n=0}^\infty}
\def\dfrac{\frac}
\def\dsum{\sum}
\def\cuadro{\hfill $\Box$}
\def\qed{\hfill $\Box$}
\def\prueba{\vskip10pt\noindent{\it PROOF.}\hskip10pt}
\newcommand{\cuadrosymb}{\mbox{ }~\hfill~{\rule{2mm}{2mm}}}

\def\L{{\mathcal L}}
\def\P{{\cal P}}
\def\B{\mathcal B}
\def\KO{{\cal K}}
\def\J{\mathcal J}

\def\M{\mathcal M}

\def\botimes{{\bf \otimes}}
\def\ba{\begin{eqnarray*}}
\def\ea{\end{eqnarray*}}

\def\be{\begin{equation}}

\def\ee{\end{equation}}

\def\A{{\bf A}}
\def\bB{{\bf B}}
\def\bT{{\bf T}}
\def\x{{\bf x}}
\def\y{{\bf y}}

\def\dt{\frac{dt}{2\pi}}
\def\ds{\frac{ds}{2\pi}}

\def\d{\displaystyle}

\def\Tkj{T_{kj}}
\def\Skj{S_{kj}}
\def\Rk{{\bf R_k}}
\def\Cj{{\bf C_j}}
\def\ej{{\bf e_j}}
\def\ek{{\bf e_k}}

\def\la{\langle}
\def\ra{\rangle}
\def\e{\varepsilon}

\def\Bl2{\mathcal B(\ell^2(H))}
\def\Ml2{\mathcal M_r(\ell^2(H))}
\def\l2{\ell^2(H)}

\def\sj{\sum_{j=1}^\infty}
\def\sk{\sum_{k=1}^\infty}

\def\fk{\varphi_k}
\def\fj{\varphi_j}

%%%%Spaces

\def\lsot{\ell^2_{SOT}(\N,\B(H))}
\def\lsott{\ell^2_{SOT}(\N^2,\B(H))}

\title{Schur product with operator-valued entries}
\author{Oscar Blasco, Ismael Garc\'{\i}a-Bayona\thanks{%
Partially supported by  MTM2014-53009-P(MINECO Spain) and  FPU14/01032 (MECD Spain)}}
\date{}
\maketitle

\begin{abstract} In this paper we characterize Toeplitz matrices with entries in the space of bounded operators on Hilbert spaces $\B(H)$ which define bounded operators   acting on $\ell^2(H)$ and  use it to get the description of the right Schur multipliers
acting on $\ell^2(H)$ in terms of certain operator-valued measures.
\end{abstract}

AMS Subj. Class: Primary 46E40, Secondary 47A56; 15B05

Key words: Schur product; Toeplitz matrix; Schur multiplier; vector-valued measures

\section{Introduction. }

Throughout the paper $X,Y$ and $E$ are complex Banach spaces  and $H$ denotes a separable complex Hilbert space  with ortonormal basis $(e_n)$. We write $\L(X,Y)$ for the space of bounded linear operators, $X^*$ for the dual space and denote $\B(X)=\L(X,X)$. We also use the notations $\ell^2(E)$, $C(\T,E)$,  $L^p(\T, E)$ or $M(\T,E)$ for the space of sequences ${\bf z}=(z_n)$ in $E$ such that $\|{\bf z}\|_{\ell^2(E)}=(\sum_{n=1}^\infty \|z_n\|^2)^{1/2}<\infty$, the space of $E$-valued continuous functions, the space of strongly measurable functions from the measure space $\T=\{z\in \C: |z|=1\}$ into $E$ with $\|f\|_{L^p(\T,E)}=(\int_0^{2\pi} \|f(e^{it})\|^p\dt)^{1/p}<\infty$ for $1\le p \le \infty$ (with the usual modification for $p=\infty$) and  the space of regular vector-valued measures of bounded variation respectively. As usual for $E=\C$ we simply write $\ell^2$, $C(\T)$, $L^p(\T)$ and $M(\T)$.

Given two matrices $A=(\alpha_{kj})$ and $B=(\beta_{kj})$ with complex entries, their Schur product  is defined by $A*B= (\alpha_{kj}\beta_{kj})$. This operation endows the space $\B(\ell^2)$  with a structure of Banach algebra. A proof of the next result, due to J. Schur, can be found in \cite[Proposition 2.1]{Be} or \cite[Theorem 2.20]{PP}.
\begin{teor} \label{ts}(Schur, \cite{Schur}) If $A=(\alpha_{kj})\in \B(\ell^2)$ and $B=(\beta_{kj})\in \B(\ell^2)$  then $A*B \in \B(\ell^2)$. Moreover
$\|A*B \|_{\B(\ell^2)}\le \|A \|_{\B(\ell^2)}\|B \|_{\B(\ell^2)}.$
\end{teor}

More generally a matrix $A=(\alpha_{kj})$ is said to be a Schur multiplier, to be denoted by $A\in {\mathcal M}(\ell^2)$, whenever
$A* B\in \B(\ell^2)$  for any $B\in \B(\ell^2)$.
 For the study of Schur multipliers we refer the reader to \cite{Be, PP}.
 Recall that  a Toeplitz matrix is a matrix $A=(\alpha_{kj})$ such that there exists a sequence of complex numbers $(\gamma_l)_{l\in \Z}$ so that $\alpha_{kj}=\gamma_{k-j}$. The study of Toeplitz matrices which define bounded operators or Schur multipliers goes back to work of Toeplitz in \cite{T}.
 The reader is referred to  \cite{AP, Be, PP} for  recent proofs of the following results  concerning Toeplitz matrices.

\begin{teor} \label{tt}(Toeplitz \cite{T}) Let $A=(\alpha_{kj})$ be a Toeplitz matrix. Then $A\in \B(\ell^2)$ if and only if there exists $f\in L^\infty(\T)$ such that $\alpha_{kj}=\hat f(k-j)$ for all $k,j\in \N$. Moreover $\|A\|=\|f\|_{L^\infty(\T)}.$
\end{teor}
\begin{teor} \label{tb}(Bennet \cite{Be}) Let $A=(\alpha_{kj})$ be a Toeplitz matrix. Then $A\in \M(\ell^2)$ if and only if there exists $\mu\in M(\T)$ such that $\alpha_{kj}=\hat\mu(k-j)$ for all $k,j\in \N$.  Moreover $\|A\|= \|\mu\|_{M(\T)}.$
\end{teor}

It is known the recent interest for operator-valued functions (see \cite{HNVW}) and for the matricial analysis (see \cite{PP}) concerning their uses in different problems in Analysis. In this paper, we would like to formulate the analogues of the theorems above in the context of matrices $\A=(\Tkj)$ with entries $\Tkj\in\B(H)$. For such a purpose we are led to consider operartor-valued measures. We shall make use of several notions and spaces from  the theory of vector-valued measures and the reader is referred to classical books  \cite{D,DU} or to \cite{B} for some new results in connection with Fourier analysis.

In the sequel we write $ \la \cdot,\cdot\ra$ and $ \ll\cdot,\cdot \gg $ for the scalar products in  $H$ and $\l2$ respectively, where
$\ll\x,\y\gg= \sum_{j=1}^\infty \la x_j,y_j\ra$ and we use the notation $x\ej=(0,\cdots,0,x,0,\cdots) $ for the element in $\l2$ in which $x\in H$ is placed in the $j$-coordinate for $j\in \N$. As usual $c_{00}(H)=span\{x\ej:x\in H, j\in \N\}$.

\begin{defi}
Given a matrix $\A =(\Tkj)$ with entries $\Tkj\in \B(H)$  and $\x\in c_{00}(H)$ we  write $\A (\x)$ for the sequence $(\sum_{j=1}^\infty\Tkj(x_j))_k$.
We say that $\A\in \B(\ell^2(H))$ if the map $\x \to \A(\x)$ extends to a bounded linear operator in $\ell^2(H)$, that is
$$\Big(\sum_{k=1}^\infty \|\sum_{j=1}^\infty \Tkj(x_j)\|^2\Big)^{1/2}\le C \Big(\sum_{j=1}^\infty \|x_j\|^2\Big)^{1/2}.$$

We shall write
$$\|\A\|_{\Bl2}=\inf\{C\ge 0: \|\A \x \|_{\l2}\le C \|\x\|_{\l2}\}.$$
\end{defi}

\begin{defi}
Given two matrices $\A =(\Tkj)$ and $\bB=(\Skj)$ with entries $\Tkj, \Skj\in \B(H)$ we define the Schur product $\A*\bB= (\Tkj\Skj)$ where $\Tkj\Skj$ stands for the composition of the operators $\Tkj$ and $\Skj$.
\end{defi}

Contrary to the scalar-valued case this product is not commutative.
\begin{defi}
Given a matrix $\A =(\Tkj)$. We say that $\A$ is a right Schur multiplier (respectively left Schur multiplier), to be denoted by $\A\in \mathcal M_r(\l2)$ (respectively $\A\in \mathcal M_l(\l2)$ ), whenever $\bB* \A\in \Bl2$
 (respectively $\A* \bB\in \Bl2$ ) for any $\bB\in \Bl2$.
We shall write
$$\|\A\|_{\Ml2}=\inf\{C\ge 0: \|\bB * \A  \|_{\Bl2}\le C \|\bB\|_{\Bl2}\}$$
and
$$\|\A\|_{\mathcal M_l(\l2)}=\inf\{C\ge 0: \|\A*\bB   \|_{\Bl2}\le C \|\bB\|_{\Bl2}\}.$$
\end{defi}

 Denoting by $\A^*$ the adjoint matrix given by  $\Skj=T^*_{jk}$ for all $k,j\in \N$,  one easily sees that $\A\in \Bl2$ if and only if $\A^*\in \Bl2$ with $\|\A\|=\|\A^*\|$ and also that $\A\in \mathcal M_l(\l2)$ if and only if  $\A^*\in \mathcal M_r(\l2)$ and $\|\A\|_{\mathcal M_l(\l2)}= \|\A^*\|_{\mathcal M_r(\l2)}$.

If $X$ and $Y$ are Banach spaces we write $X\hat\otimes Y$ for the projective tensor product. We refer the reader to \cite[Chap.8]{DU}, \cite[Chap.2]{Ryan}) or \cite{DFS} for all possible results needed in the paper. We recall that $(X\hat\otimes Y)^*=\L(X,Y^*)$ and to avoid misunderstandings, for each $T\in \L(X,Y^*)$, we write $\J T$ when $T$ is seen as an element in $(X\hat\otimes Y)^*$. In other words we write $\mathcal J:\L(X,Y^*)\to (X\hat\otimes Y)^*$  for the isometry given by $ \mathcal J T(x\otimes y)=T(x)(y)$ for any $T\in \L(X,Y^*)$, $x\in X$ and $y\in Y$. Also, given $x^*\in X^*$ and $y^*\in Y^*$, we write $\widetilde{x^*\otimes y^*}$ for the operator in $\L(X, Y^*)$ given by $\widetilde{x^*\otimes y^*}(z)=x^*(z) y^*$ for each $z\in X$.
In the paper we shall restrict ourselves to the case $\L(X,Y^*)=\B(H)$, that is $X=Y^*=H$. Using the Riesz theorem  we identify $Y=Y^*=H$. Hence for $T,S\in \B(H)$ and $x,y\in H$ we shall use the following formulae \be \la  T(x),y\ra= \J T(x\botimes y),  \ee
\be \widetilde{(x\otimes y)}(z)=\la z,x\ra y, \quad z\in H,\ee
 \be \label{compo}T\widetilde{(x\otimes y)}=\widetilde{(x\otimes (Ty))}, \quad
\widetilde{(x\otimes y)}T=\widetilde{(T^*x)\otimes y},
\ee
\be \label{mainformula}
\J (TS)(x\otimes y)=\J T(Sx\otimes y)=\J S(x\otimes T^*y).
\ee

The paper is divided into four sections. The  first  section is of a preliminary character and we recall the basic notions on vector-valued sequences and functions to be used in the sequel. Next section contains several results on regular operator-valued measures which are the main ingredients for the remaining proofs in the paper. In Section 4 we are concerned with several necessary and sufficient conditions for a matrix $\A$ to belong to $\Bl2$ and we show that the Schur product endows $\Bl2$ with a Banach algebra structure also in this case. The final section deals with Toeplitz matrices
$\A$ with entries in $\B(H)$, that is those matrices  for which there exists a sequence $( T_l)_{l\in \Z}\subset \B(H)$ so that $\Tkj=T_{k-j}$.  We shall write $\mathcal T$ the family of such Toeplitz matrices and
we  characterize $\mathcal T\cap \Bl2$ as those matrices where $T_l=\hat\mu(j-k)$ for a certain regular operator-valued vector measure  $\mu$ belonging to $V^\infty(\T, \B(H))$ (see Definition \ref{vinf} below). Concerning the analogue of Theorem \ref{tb}  we shall show that $M(\T,\B(H))\subseteq \Ml2\subseteq M_{SOT}(\T,\B(H))$  where $M(\T,\B(H))$ stands for the space of regular operator-valued measures and $M_{SOT}(\T,\B(H))$ is defined, using the strong operator topology, as the space of vector measures $\mu$ such that $\mu_x\in M(\T, H)$ given by $\mu_x(A)=\mu(A)(x)$ for any $x\in H$.

\section{Preliminaries on operator-valued sequences and  functions.}

Write  $\ell^2_{weak}(\N, \B(H))$ and $\ell^2_{weak}(\N^2, \B(H))$ for the space of sequences ${\bf T}=(T_n)\subset \B(H)$ and matrices $\A=(T_{kj})\subset \B(H)$ such that
$$\|{\bf T}\|_{\ell^2_{weak}(\N, \B(H))}=\sup_{\|x\|=1, \|y\|=1}(\sum_{n=1}^\infty |\la T_n(x),y\ra|^2)^{1/2}<\infty$$ and
$$\|\A\|_{\ell^2_{weak}(\N^2, \B(H))}=\sup_{\|x\|=1, \|y\|=1}(\sk\sj |\la T_{kj}(x),y\ra|^2)^{1/2}<\infty.$$
The reader can see that these spaces actually coincide with the ones appearing using notation in \cite{DJT}.
Of course $\ell^2(E)\subsetneq \ell^2_{weak}(E)$. In the case $\B(H)$ we can actually  introduce certain spaces between $\ell^2(E)$ and $\ell^2_{weak}(E)$.
\begin{defi}Given a sequence ${\bf T}=(T_n)$  and a matrix $\A=(\Tkj)$ of operators in $\B(H)$ we write
\be \label{sot1}\|{\bf T}\|_{\lsot}=\sup_{\|x\|=1} (\sum_{n=1}^\infty \|T_n(x)\|^2)^{1/2}\ee
and
\be \label{sot2}\|{\bf A}\|_{\lsott}=\sup_{\|x\|=1} (\sj\sk \|\Tkj(x)\|^2)^{1/2}.\ee
We
set $\lsot$ and $\lsott$ for the spaces of sequences and operators with $\|{\bf T}\|_{\lsot}<\infty$ and $\|{\bf A}\|_{\lsott}<\infty$ respectively.
\end{defi}

\begin{nota} It is easy to show that $$\ell^2(\N^2, \B(H))\subsetneq \ell^2(\N, \ell^2_{SOT}(\N, \B(H))\subsetneq \lsott.$$
\end{nota}

As usual  we denote $\fk(t)=e^{ikt}$ for $k\in \Z$, and, given a complex Banach space $E$,  we write $\P(\T, E)= span\{e\fj: j\in \Z, e\in E\}$ for the $E$-valued trigonometric polynomials,
$\P_a(\T, E)= span\{e\fj: j\in \N, e\in E\}$ for the $E$-valued analytic polynomials. It is well-known that $\P(\T,E)$ is dense in $C(\T,E)$ and $L^p(\T,E)$ for $1\le p<\infty$. Also we shall use $H_0^2(\T, E)=\{f\in L^2(\T, E): \hat f(k)=0, k\le 0\}$, where $\hat f(k)=\int_0^{2\pi} f(t)\overline{ \fk(t)}\frac{dt}{2\pi}$  for $k\in \Z$. Recall that $H_0^2(\T, E)$ coincides with  the closure of  $\P_a(\T, E)$ with the norm in $L^2(\T, E)$.
Similarly $H_0^2(\T^2, E)=\{f\in L^2(\T^2, E): \hat f(k,j)=0,\quad  k,j\le 0\}$, where $\hat f(k,j)=\int_0^{2\pi}\int_0^{2\pi} f(t,s) \overline{ \fk(t)}\overline{  \fj(s)}\frac{dt}{2\pi}\ds$  for $k,j\in \Z$.

Let us now introduce some new spaces that we shall need later on.
\begin{defi} Let ${\bf T}=(T_n)\subset \B(H)$ and $\A=(\Tkj)\subset \B(H)$. We say that ${\bf T}\in \tilde H^2(\T, \B(H))$ whenever
$$\|{\bf T}\|_{\tilde H^2(\T, \B(H))}=\sup_N (\int_0^{2\pi}\|\sum_{n=1}^NT_n\fj(t)\|^2\dt)^{1/2}<\infty.$$
We say that $\A\in \tilde H^2(\T^2, \B(H))$ whenever
$$\|{\A}\|_{\tilde H^2(\T^2, \B(H))}=\sup_{N,M} (\int_0^{2\pi}\int_0^{2\pi}\|\sum_{j=1}^N\sum_{k=1}^M\Tkj\fj(t)\fk(s)\|^2\dt\ds)^{1/2}<\infty.$$
\end{defi}

\begin{nota} \label{nota1} $\tilde H^2(\T, \B(H))\nsubseteq H_0^2(\T, \B(H))$.

Consider $T_j=\widetilde{e_j\otimes e_j}$. Then
 for any $t\in [0,2\pi)$ and $N\in \N$,
$$\|\sum_{j=1}^N\widetilde{(e_j\otimes e_j})\fj(t)\|_{\B(H)}=\sup_{\|x\|=1}\|\sum_{j=1}^N\la x,e_j\ra\fj(t)e_j\|_H=1.$$
Hence
${\bf T}=(e_j\otimes e_j)_j\in \tilde H^2(\T, \B(H))$, but since $\|T_j\|=1$ then  ${\bf T}\notin  H_0^2(\T, \B(H)$ (since even ${\bf T}\notin L^1(\T, \B(H))$ because $\lim_{j\to\infty}\|T_j\|\neq 0$).
\end{nota}

\begin{prop}

(i) $\tilde H^2(\T, \B(H))\subsetneq \lsot$ and $\tilde H^2(\T^2, \B(H))\subsetneq \lsott$.

(ii) $\tilde H^2(\T, \B(H))\nsubseteq \ell^2(\N, \B(H))$ and $\ell^2(\N, \B(H)) \nsubseteq \tilde H^2(\T, \B(H))$.
 \end{prop}
 \begin{proof}
 (i) Both inclusions are immediate from Plancherel's theorem (which holds for Hilbert-valued functions). It suffices to see that there exists ${\bf T}\in \lsot\setminus \tilde H^2(\T, \B(H))$ because choosing matrices with a single row we obtain also a counterexample for the other inclusion.
 Now selecting $T_n=\widetilde{e_n\otimes x}\in \B(H)$ for a given $x\in H$ we clearly have ${\bf T}=(\widetilde{e_n\otimes x})_n\in \lsot$ with $\|{\bf T}\|_{\lsot}=\|x\|$. However, for any $t\in [0,2\pi)$ and $N\in \N$,
$$\|\sum_{n=1}^N(\widetilde{e_n\otimes x})\varphi_n(t)\|_{B(H)}=\|\widetilde{(\sum_{n=1}^Ne_n\varphi_n(t))\otimes x}\|_{\B(H)}=\|x\|\sqrt N,$$
showing that ${\bf T}\notin \tilde H^2(\T, \B(H)).$

(ii)
The example in Remark \ref{nota1}  shows that $\tilde H^2(\T, \B(H))\nsubseteq \ell^2(\N, \B(H))$.
Let us  now find ${\bf T}\in \ell^2(\N, \B(H))\setminus \tilde H^2(\T, \B(H))$.
Consider $H=L^2(\T)$ and ${\bf T}=(T_j)$ where $T_j: L^2(\T)\to L^2(\T)$ is given by
$T_j( f)= \frac{\fj}{j}f.$

 Clearly ${\bf T}\in \ell^2(\N, \B(H))$ since $\|T_j\|=\frac{1}{j}$ for all $j\in \N$. On the other hand, for each $t\in [0,2\pi)$ and $N\in \N$ one has that $(\sum_{j=1}^NT_j\fj(t))(f)= (\sum_{j=1}^N\frac{\fj(t)}{j}\fj)f$ and therefore
$$\|\sum_{j=1}^NT_j\fj(t)\|_{B(H)}= \|\sum_{j=1}^N\frac{\fj(t)}{j}\fj\|_{C(\T)}=
\sum_{j=1}^N\frac{1}{j}.$$
This shows that ${\bf T}\notin \tilde H^2(\T, \B(H)).$
\end{proof}

\section{Preliminaries on regular vector measures}

We recall some facts for vector measures that can  be found in \cite{DU, D}.
Let us consider the measure space $(\T, {\mathfrak B}(\T), m)$ where ${\mathfrak  B}(\T)$ stands for the Borel sets over $\T$ and $m$ for the Lebesgue measure on $\T$. Given a vector measure $\mu:{\mathfrak B}(\T)\to E$ and $B\in {\mathfrak B}(\T)$, we shall denote $|\mu|(B)$ and $\|\mu\|(B)$ the variation and semi-variation of $\mu$ of the set $B$ given by
$$|\mu|(B)=\sup\{ \sum_{A\in\pi}\|\mu(A)\|, A\in {\mathfrak B}(\T), \pi \hbox{ finite partition of } B\}$$
and
$$\|\mu\|(B)=\sup\{ |\la e^*,\mu\ra|(B): e^*\in E^*, \|e^*\|=1\}$$
 where $\la e^*,\mu\ra(A)=e^*(\mu(A))$ for all $A\in {\mathfrak B}(\T)$. Of course $|\mu|(\cdot)$ becomes a positive measure on $\mathfrak B(\T)$, while $\|\mu \|(\cdot)$ is only sub-additive in general. We shall denote $|\mu|=|\mu|(\T)$ and $\|\mu\|=\|\mu\|(\T)$.
 For dual spaces $E=F^*$ it is easy to see that
 $\|\mu\|=\sup\{ |\la\mu, f\ra|: f\in F, \|f\|=1\}$
 where $\la \mu,f\ra(A)=\mu(A)(f)$.

  In what follows we shall consider regular vector measures, that is to say vector measures $\mu:{\mathfrak B}(\T)\to E$  such that for each $\varepsilon>0$ and $B\in {\mathfrak B}(\T)$ there exists a compact set $K$, an open set $O$ such that $K\subset B\subset O$ with $\|\mu\|(O\setminus K)<\varepsilon$.   Let us denote by ${\mathfrak M}(\T, E)$ and $M(\T, E)$ the spaces of regular Borel measures with values in $E$ endowed with the norm given the semi-variation and variation respectively.
 Of course $M(\T, E)\subsetneq {\mathfrak M}(\T, E)$ when $E$ is infinite dimensional.

  It is well known that the space ${\mathfrak M}(\T, E)$   can be identified with  the space of weakly compact  linear operators $T_{\mu}:C(\T)\to E$ and that $\|T_{\mu}\|= \|\mu\|$ (see \cite[Chap. 6]{DU}).
Hence, for each  $\mu\in {\mathfrak M}(\T, E)$ and $k\in \Z$ we can define (see \cite{B}) the $k$-Fourier coefficient by
 \be\label{fcd}\hat\mu(k)=T_\mu(\fk).
 \ee

Also the description of measures in $M(\T,E)$ can be done using absolutely summing operators (see \cite{DJT}) and the variation can be described as the norm in such space (see \cite{DU}) but we shall not follow this approach. On the other hand since we deal with either $E=\B(H)$ or  $E=H$ we have at our disposal Singer's theorem (see for instance \cite{S, S2, H}), which in the case of dual spaces $E=F^*$ asserts that $M(\T, E)=C(\T, F)^*$. In other words there exists a bounded map $\Psi_\mu:C(\T, F)\to \C$ with $\|\Psi_\mu\|=|\mu|$ such that
 \be
 \Psi_\mu(y\phi)= T_\mu(\phi)(y), \quad \phi\in C(\T), y\in F.
 \ee
 In particular for  $k\in \Z$ one has $\hat\mu(k)(y)=\Psi_\mu(y\fk)$  for each $y\in F$.

As mentioned above since
$ M(\T, \L(X,Y^*))=C(\T, X\hat\otimes Y)^*$, for each $\mu\in M(\T, \L(X,Y^*))$ we can associate two operators $T_\mu$ and $\Psi_\mu$. Of course the connection between them  is given by the formula
\be \label{mea1}
 T_{\mu}(\phi)(x)(y)=\Psi_{\mu}( (x\otimes y)\phi), \quad \phi\in C(\T), x\in X,y\in Y.
\ee

There is still one more possibility to be considered using the strong operator topology, namely $\Phi_{\mu}: C(\T,X) \to Y^*$ defined by
\be
\Phi_{\mu}(f)(y)=\Psi_{\mu}(f\otimes y), \quad f\in C(\T,X), y\in Y
\ee
where $f\otimes y(t)=f(t)\otimes y$.

Therefore given $\mu\in {\mathfrak M}(\T, \L(X,Y^*))$ we have three different linear operators defined on the corresponding spaces of polynomials: $T_\mu: \P(\T)\to \L(X,Y^*)$, $\Psi_\mu: \P(\T, X\hat\otimes Y)\to \C$ and $\Phi_{\mu}: \P(\T,X) \to Y^*$ defined by the formulae
\be \label{formula1}
T_\mu(\sum_{j=-M}^N\alpha_j\fj)= \sum_{j=-M}^N\alpha_{j}\hat\mu(j), \quad N,M\in \N, \alpha_j\in \C,
\ee
\be \label{formula2}
\Psi_{\mu}(\sum_{j=-M}^N (\sum_{n=1}^{n_j}x_{jn})\otimes (\sum_{m=1}^{m_j}y_{jm})\fj)= \sum_{j=-M}^N(\sum_{n=1}^{n_j} \sum_{m=1}^{m_j} \hat\mu(j)(x_{jn})(y_{jm})),
\ee
\be \label{formula3}
 \Phi_{\mu}(\sum_{j=-M}^N x_j\fj)= \sum_{j=-M}^N\hat\mu(j)(x_j), \quad N,M\in \N, x_j\in X.
\ee

When restricting to the case $Y^*=H$ we obtain the following connection between them.
\be\label{conn} \J T_{\mu}(\psi)(x\otimes y)= \Psi_{\mu}((x\otimes y)\psi)=\la  \Phi_{\mu}(x\psi),y\ra, \quad \psi\in \P(\T), x,y\in H.\ee

Given $\mu\in {\mathfrak M}(\T, \L(X,Y^*))$ and  $x\in H$ let us denote,  by $\mu_x$ the $Y^*$-valued  measure given by $$\mu_x(A)=\mu(A)(x), \quad A\in {\mathfrak B}(\T).$$ It is elementary to see that $\mu_x$ is a regular measure because one can associate the weakly compact operator $T_{\mu_x}=\delta_x\circ T_{\mu}:C(\T)\to Y^*$ where  $\delta_x$ stands for the operator  $\delta_x:\L(X,Y^*)\to Y^*$ given by $\delta_x(T)=T(x)$ for  $T\in \L(X,Y^*)$.

If $\mu\in {\mathfrak M}(\T, \B(H))$, $k\in \Z$ and $x,y\in H$ then $\mu_x\in {\mathfrak M}(\T,H)$, \be\label{mux} \la\mu_x(A), y\ra= \J\mu(A)(x\otimes y), \quad A\in \mathfrak B(\T)\ee
and
\be\label{fc}
\la\hat\mu(k)(x),y\ra=\la \hat\mu_x(k),y\ra= \J\hat\mu(k)(x\otimes y).
\ee

Let us introduce a new space of measures appearing in the case $E=\B(H)$.
\begin{defi} Let $\mu \in {\mathfrak M}(\T,\B(H))$. We say that $\mu\in M_{SOT}(\T,\B(H))$ if $\mu_x\in M(\T, H)$ for any $x\in H$.
We write
$$\|\mu\|_{SOT}=\sup\{ |\mu_x|: x\in H, \|x\|=1\}.$$
\end{defi}

\begin{prop} $M(\T, \B(H))\subsetneq M_{SOT}(\T, \B(H))\subsetneq {\mathfrak M}(\T, \B(H)).$
\end{prop}
\begin{proof}  The inclusions between the spaces follow from the inequalities
$$|\la \mu(A)(x), y\ra|\le \|\mu(A)(x)\|\|y\|\le \|\mu(A)\|\|x\|\|y\|$$
which leads to
$$|\la \mu_x, y\ra| \le |\mu_x|\|y\|\le |\mu|\|x\|\|y\|$$
and the corresponding embeddings with norm $1$ trivially follow.

Let $H=\ell^2$. We shall find   measures $\mu_1\in M_{SOT}(\T,\B(H))\setminus M(\T,\B(H))$ and $\mu_2\in {\mathfrak M}(\T, \B(H))\setminus M_{SOT}(\T, \B(H))$.
Both can be constructed relying on a similar argument.
 Let $y_0\in H$ with $\|y_0\|=1$ and select a Hilbert-valued regular measure  $\nu$  with $|\nu|=\infty$ (for instance take a Pettis integrable, but not Bochner integrable function $f:\T\to H$ given by $t\to (f_n(t))_n$ and $\nu(A)=(\int_A f_n(t)\dt)_n$ for  $A\in {\mathfrak B}(\T)$). Denote $T_\nu:C(\T)\to H$ the corresponding bounded (and hence weakly compact) operator associated to $\nu$ with $\|T_\nu\|=\|\nu\|$.

Define $$\mu_1(A)(x)=\la x, \nu(A)\ra y_0, \quad A\in {\mathfrak B}(\T)$$
and
$$\mu_2(A)(x)=\la x, y_0\ra \nu(A), \quad A\in {\mathfrak B}(\T).$$

In other words, if $J_y:H\to \B(H)$ and $I_y:H\to \B(H)$ stand for the operators $$J_y(x)(z)=\la z,x\ra y, \quad I_y(x)(z)=\la x,y\ra z, \quad x,y,z\in H$$
then we have that $T_{\mu_1}=J_{y_0}T_\nu$ and $T_{\mu_2}=I_{y_0}T_\nu$ are weakly compact. Hence  $\mu_1,\mu_2\in {\mathfrak M}(\T,\B(H))$.

Note that $|(\mu_1)_x |=|\la x,\nu\ra|$ and $|(\mu_2)_x |=|\la x,y_0 \ra| |\nu|, \quad x\in H$. Hence
$$\|\mu_1\|_{SOT}= \|\nu\|, \quad \|\mu_2\|_{SOT}=|\nu|.$$
Also notice that $\|\mu_1(A)\|_{\B(H)}=\|\nu(A)\|_H$, and therefore
$|\mu_1|= |\nu|,$
which gives the desired results.
\end{proof}

\begin{defi}

Let $\mu:{\mathfrak B}(\T)\to \L(X,Y^*)$ be a vector measure. We define ``the adjoint measure" $\mu^*:{\mathfrak B}(\T)\to \L(Y, X^*)$  by the formula
\be \label{dualmes0}
\mu^*(A)(y)(x)= \mu_x(A)(y), \quad A\in {\mathfrak B}(\T), x\in X, y\in Y.
\ee
\end{defi}

In the case that $\mu\in {\mathfrak M}(\T, \B(H))$  with the identification $Y^*=H$, one clearly has that
\be \label{dualmes}
\la x, \mu^*(A)(y)\ra= \la \mu(A)(x),y\ra, \quad A\in {\mathfrak B}(\T), x,y\in H.
\ee

 \begin{nota}  $\mu^*$ belongs to ${\mathfrak M}(\T, \B(H))$ (respect. $ M(\T, \B(H))$) if and only if $\mu$ belongs to ${\mathfrak M}(\T, \B(H))$ (respect. $M(\T, \B(H)))$). Moreover $\|\mu\|=\|\mu^*\|$ (respect.$|\mu|=|\mu^*|$.)

  The results follow using that  $T_{\mu^*}(\phi)= (T_{\mu}(\phi))^*$ for any $\phi\in C(\T)$ and $\|\mu(A)\|=\|\mu^*(A)\|$ for any $A\in {\mathfrak B}(\T)$.

\end{nota}
Let us describe the norm in $M_{SOT}(\T, \B(H)$ using the adjoint measure.

\begin{prop} Let $\mu\in {\mathfrak M}(\T, \B(H))$. Then
$\mu\in M_{SOT}(\T, \B(H))$ if and only if $\Phi_{\mu^*}\in \L(C(\T,H),H)$.
Moreover $\|\mu\|_{SOT}=\|\Phi_{\mu^*}\|.$
\end{prop}
\begin{proof} By definition $\mu \in M_{SOT}(\T, \B(H))$ if and only if the operator $S_\mu(x)= \mu_x$ is well defined and belongs to $\L(H,M(\T,H))$.  Moreover $\|\mu \|_{SOT} =\|S_\mu\|$. The result follows if we show that $S_\mu$ is the adjoint of $\Phi_{\mu^*}$. Recall that, identifying  $H=H^*$, we have $\mu^*\in {\mathfrak M}(\T, \B(H))$.  Hence $\Phi_{\mu^*}: \P(\T,H)\to H$ is generated by linearity using
$$\Phi_{\mu^*}(x\fk)=\widehat{\mu^*}(k)(x)= \hat{\mu}(k)^*(x) , \quad x\in H, k\in \Z.$$ Therefore, if $k\in \Z$, $x,y\in H$, since $M(\T,H)=(C(\T,H))^*$, we have
$$S_\mu(y)( x\fk)=\Psi_{\mu_y}(x\fk)=\la \widehat{\mu_y}(k),x\ra= \la \hat{\mu}(k)(y),x\ra=\la y,\Phi_{\mu^*}(x\fk)\ra.$$
By linearity we extend to $\la y,\Phi_{\mu^*}(x\phi)\ra= S_\mu(y)( x\phi)$ for any polynomial $\phi$ and since $\P(\T,H)$ is dense in $C(\T,H)$ we obtain the result. This completes the proof.
\end{proof}

 Let us consider  the following subspace of regular measures which plays an important role in what follows.

\begin{defi}\label{vinf} Let us write $V^\infty(\T,E)$ for the subspace   of those measures $\mu\in {\mathfrak M}(\T, E)$ such that there exists $C>0$ with \be \label{vinffor} \|\mu(A)\|\le C m(A), \quad  A\in {\mathfrak B}(\T).\ee
 We define
  $$\|\mu\|_\infty=\sup\{ \frac{\|\mu(A)\|}{m(A)}: m(A)>0\}.$$
  \end{defi}
  It is clear that any $\mu\in V^\infty(\T,\B(H))$ also belongs to $M(\T, \B(H))$ and it is absolutely continuous with respect to  $m$.

Let us point out two more possible descriptions of $V^\infty(\T,E)$. One option is to look at $V^\infty(\T,E)=\L(L^1(\T), E)$ (see \cite[page 261]{D}), that is to say that $T_{\mu}$ has a bounded extension to $L^1(\T)$. Hence a measure $\mu\in {\mathfrak M}(\T, E)$ belongs to $V^\infty(\T, E)$ if and only if \be\|T_\mu(\psi)\|\le C \|\psi\|_{L^1(\T)}, \quad \psi\in C(\T).\ee Moreover
  $\|T_{\mu}\|_{L^1(\T)\to E}=\|\mu\|_\infty.$

 In the case that $E=F^*$ also one has that  $V^\infty(\T, E)=L^1(\T,F)^*$, that is the dual of the space of Bochner integrable functions.  In this case a measure $\mu\in V^\infty(\T, E)$ if and only if $\Psi_{\mu}$ has a bounded extension to $L^1(\T, F)^*$, that is
 \be\|\Psi_\mu(p)\|\le C \|p\|_{L^1(\T, F)}, \quad p\in \P(\T,  F ).\ee Moreover
  $\|\Psi_{\mu}\|_{L^1(\T,F)^*}=\|\mu\|_\infty.$

Although measures in $V^\infty(\T,\B(H))$ are absolutely continuous with respect to $m$, the reader should be aware that they might not have a Radon-Nikodym derivative in $L^1(\T,E)$ (see \cite[Chap. 3]{DU}).

 For the sake of completeness we give an example for $E=\B(H)$ of such a situation.

 \begin{prop}
 Let $H=\ell^2$ and $\mu\in {\mathfrak M}(\T, \B(H))$ such that
$T_\mu\in \L( C(\T), \B(H))$ is given by $$T_\mu(\phi)=\sum_{n=1}^\infty\hat\phi(n)\widetilde{e_n\otimes e_n}.$$
Then $\mu\in  V^\infty(\T, \B(H))$ with $\|\mu\|_\infty=1$, $$\hat\mu(k)=\left\{
                                                                           \begin{array}{ll}
                                                                             \widetilde{e_k\otimes e_k}
& k\ge 1 \\
                                                                             0, & k\le 0
                                                                           \end{array}
                                                                         \right.
 $$ but it does not have a Radon-Nikodym derivative in $L^1(\T, \B(H))$.
\end{prop}
\begin{proof}
Let us show that $T_\mu$ defines a continuous operator from $L^1(\T)$ to $\B(H)$ with norm $1$. In such a case using that the inclusion $C(\T)\to L^1(\T)$ is weakly compact  one automatically has that $\mu\in {\mathfrak M}(\T, \B(H))$.  For $x=(\alpha_n)\in H$ and $y=(\beta_n)\in H$ one has
\ba
|\la T_\mu(\phi)(x),y\ra|&=&| \sum_{n=1}^\infty\hat\phi(n)\alpha_n\beta_n|\\
&\le& \sup_{n\ge 1}|\hat\phi(n)|\|x\|\|y\|\\
&\le&\|\phi\|_{L^1(\T)}\|x\|\|y\|.
\ea
This gives that $\mu\in  V^\infty(\T, \B(H))$ and $\|\mu\|_\infty\le 1$.
Using that $T_{\mu}(\fj)=\widetilde{e_j\otimes e_j}$ and $\|\widetilde{e_j\otimes e_j}\|_{\B(H)}=1$ we get the equality of norms.

The result on Fourier coefficients is obvious. To show  that $\mu$ does not have a Bochner integrable  Radon-Nikodym derivative follows now using that otherwise $\hat\mu(k)=\hat f(k)$  for some $f\in L^1(\T, \B(H))$ which implies that $\|\hat f(k)\|\to 0$ as $k\to\infty$ while  $\|\hat \mu(k)\|=1$ for $k\ge 1$. This completes the proof.
\end{proof}

We finish this section with a known characterization of measures in $M(\T,F^*)$ to be used later on, that we include for sake of completeness.
\begin{lema} \label{convo} Let $E=F^*$ be a dual Banach space and $\mu\in {\mathfrak M}(\T,E)$. For each $0<r<1$ we define
\be \label{conv}
P_r*\mu(t)=\sum_{k\in \Z} \hat\mu(k) r^{|k|}\fk(t), \quad t\in [0,2\pi).
\ee
Then

(i) $P_r*\mu\in C(\T,E)$  and $\|P_r*\mu\|_{ C(\T,E)}\le \|\mu\|\frac{1+r}{1-r}$ for any $0<r<1$.

(ii) $\mu\in M(\T, E)$ if and only if $\sup_{0<r<1} \|P_r*\mu\|_{L^1(\T,E)}<\infty.$
Moreover $$|\mu|= \sup_{0<r<1}\|P_r*\mu\|_{L^1(\T,E)}.$$
\end{lema}
\begin{proof} (i)  Observe that
$$\sum_{\in \Z} |\hat\mu(k)| r^{|k|}\|\fk\|_{C(\T)}\le \|T_\mu\|(1+ 2\sum_{k=1}^\infty r^k)=\|\mu\|\frac{1+r}{1-r}.$$
This shows that the series in (\ref{conv}) is absolutely convergent in $C(\T,E)$ and we obtain (i).

(ii) Assume that $\mu\in M(\T, E)$. In particular $|\mu|\in M(\T)$ and
$$\int_0^{2\pi} \|P_r*\mu(t)\|\dt \le \int_0^{2\pi} P_r*|\mu|(t)\dt.$$
Hence, using  the scalar-valued result, we have
$$\sup_{0<r<1}\|P_r*\mu\|_{L^1(\T,E)}\le \sup_{0<r<1}\|P_r*|\mu|\|_{L^1(\T)}\le \sup_{0<r<1}|\mu| \|P_r\|_{L^1(\T)}=|\mu|.$$
Conversely, assume that $\sup_{0<r<1} \|P_r*\mu\|_{L^1(\T,E)}<\infty$. Since $L^1(\T,E)\subseteq M(\T, E)=C(\T, F)^*$, from the Banach-Alaoglu theorem one can find a sequence $r_n$ converging to $1$ and a measure $\nu\in M(\T, E)$ such that $P_{r_n}*\mu \to \nu$ in the $w^*$-topology. Selecting now functions in $C(\T, F)$ given by $y\fk$ for all $y\in F$ and $k\in \Z$ one shows that $\hat\nu(k)=\hat\mu(k)$. This gives that  $\mu=\nu$ and therefore $\mu\in M(\T,E)$.
Finally, notice that $$|\mu|= \sup \{|\Psi_\mu(p)|: p\in \P(\T,F), \|p\|_{C(\T,F)}=1\}.$$
Given now $p=\sum_{k=-M}^N y_k \fk$, one has  $P_r*p= \sum_{k=-M}^N y_k r^{|k|}\fk$ and
$$\Psi_\mu(P_r*p)= \sum_{k=-M}^N \hat\mu(k)(y_k) r^{|k|}=\int_0^{2\pi} P_r*\mu(t)(p(t))\dt.$$
 Finally since $p=\lim_{r\to 1} P_r*p$ in $C(\T, F)$ then
\ba|\Psi_\mu(p)|&=& \lim_{r\to 1} |\Psi_\mu(P_r*p)|\\
&\le& \sup_{0<r<1}|\int_0^{2\pi} P_r*\mu(t)(p(t))\dt|\\
&\le& \sup_{0<r<1}\| P_r*\mu\|_{L^1(\T,E)}\|p\|_{C(\T,F)}.
\ea
This gives the inequality $|\mu|\le \sup_{0<r<1}\| P_r*\mu\|_{L^1(\T,E)}$ and the proof is complete.
\end{proof}

\section{Some results on  matrices of operators.}

 Throughout the rest of the paper we write
 $\A=(\Tkj)\subset \B(H)$, $\Rk$ and $\Cj$   the $k$-row respectively, that is
$$\Rk=(\Tkj)_{j=1}^\infty, \quad \Cj=(\Tkj)_{k=1}^\infty$$
 and
\be \label{funNM}
\A_{N,M}(s,t)=\sum_{k=1}^M\sum_{j=1}^N \Tkj\overline{\fj(s)}\fk(t), \quad 0\le t,s< 2\pi, \quad N,M\in \N.
\ee

For each $\x=(x_j)\in \l2$ we  consider the functions $h_{\x}$ and $F_{\x}$ given by  \be\label{fun1}h_{\x}(t)=\sj x_j\fj(t), \quad t\in[0,2\pi).\ee

\begin{nota}
Observe that $\A\in \tilde H^2(\T^2, \B(H))$ if and only if
$$\sup_{N,M} \|\A_{N,M}\|_{L^2(\T^2, \B(H))}<\infty.$$

 Note that $\x\in \l2$ if and only if $h_\x\in H^2_0(\T,H)$.  Moreover $$\|\x\|_{\l2}=\|h_{\x}\|_{H^2(\T,H)}.$$
\end{nota}

\begin{prop} Let $\A=(\Tkj)\subset \B(H)$.

(i) If $\A\in \lsott$ then $\Rk, \Cj\in \lsot$  for  all $k,j\in \N$.

(ii) If $\A\in \tilde H^2(\T^2, \B(H))$ then $\Cj, {\bf R_k}\in \tilde H^2(\T, \B(H))$ for all $j,k\in \N$.
\end{prop}
\begin{proof}
(i) follows trivially from the definitions.

(ii) Let $k'\in \N$, $M\in \N$ and $t\in [0,2\pi)$. For $N\ge k'$ we have
$$\sum_{j=1}^N T_{k'j}\fj(t)= \int_0^{2\pi}\Big(\sum_{k=1}^N\sum_{j=1}^M T_{kj}\fj(t)\fk(s)\Big)\overline{\varphi_{k'}(s)} \ds.$$
Therefore
$$\int_0^{2\pi}\|\sum_{j=1}^N T_{k'j}\fj(t)\|^2\dt\le \int_0^{2\pi}\int_0^{2\pi}\|\sum_{k=1}^N\sum_{j=1}^M T_{kj}\fj(t)\fk(s)\|^2 \ds\dt.$$
Hence $\|{\bf R_{k'}}\|_{\tilde H^2(\T, \B(H))}\le \|\A\|_{\tilde H^2(\T^2, \B(H))}.$
A similar argument shows that $\|\Cj\|_{\tilde H^2(\T, \B(H))}\le \|\A\|_{\tilde H^2(\T^2, \B(H))}$ and it is left to the reader.
\end{proof}

\begin{defi} Let  $\A=(\Tkj)\subset \B(H)$.
Define
 $B_\A:\P_a(\T,H)\times \P_a(\T,H)\to \C$ be given  by
\be\label{bil}
(h_\x, h_\y)\to \int_0^{2\pi} \int_0^{2\pi}\J \A_{N,M}(s,t)(h_\x(s)\otimes h_\y(t))\ds\dt,
\ee
where $h_\x=\sum_{j=1}^Nx_j\fj$ and $h_\y=\sum_{k=1}^My_k\fk$ for $x_j,y_k\in H$.
\end{defi}

We now give the characterization of bounded operators in $\Bl2$ in terms of bilinear maps.

\begin{prop} \label{l3} If $\A=(\Tkj)\subset \B(H)$ then
\be\label{f2} \ll \A(\x),\y\gg = B_\A(h_\x,h_\y), \quad \x,\y\in c_{00}(H).\ee
In particular, $\A\in \Bl2$ if and only if $B_\A$ extends to a bounded bilinear map on $H^2_0(\T,H)\times H^2_0(\T,H)$. Moreover $\|\A\|=\|B_\A\|$.
\end{prop}
\begin{proof}
To show (\ref{f2}) we observe that for $h_\x=\sum_{j=1}^Nx_j\fj$ and $h_\y=\sum_{k=1}^My_k\fk$ we have $y_k=\int_0^{2\pi}h_{\y}(t)\overline{\fk(t)}\dt$ and
$x_j=\int_0^{2\pi}h_{\x}(t)\overline{\fj(s)}\ds$.  Hence
\ba
\sum_{k=1}^M \la   \sum_{j=1}^N \Tkj x_j,y_k\ra&=& \int_0^{2\pi} \la \sum_{k=1}^M (\sum_{j=1}^N \Tkj x_j)\fk(t),h_\y(t)\ra\dt\\
&=& \int_0^{2\pi} \la \sum_{k=1}^M (\sum_{j=1}^N \Tkj\fk(t))(x_j),h_\y(t)\ra\dt\\
&=& \int_0^{2\pi} \la \int_0^{2\pi}\A_{N,M}(s,t) (h_\x(s))\ds,h_\y(t)\ra\dt\\
&=&\int_0^{2\pi} \int_0^{2\pi}\J \A_{N,M}(s,t)(h_\x(s)\otimes h_\y(t))\ds\dt.
\ea
The equality of norms follows trivially.
\end{proof}

From Proposition \ref{l3} one can produce some sufficient conditions for $\A$ to belong to $\Bl2$.

\begin{coro}  If $\A\in \tilde H^2(\T^2, \B(H))\cup \ell^2(\N^2, \B(H))$ then $\A\in \Bl2$ and $\|\A\|\le \min\{\|\A\|_{\tilde H^2(\T^2, \B(H))}, \|\A\|_{\ell^2(\N,\B(H))}\}.$
\end{coro}
\begin{proof}  Assume first $\A\in  \ell^2(\N^2, \B(H))$. Then
$$|\ll \A(\x), \y\gg | \le \sk\sj \|\Tkj\|\|x_j\|\|y_k\|$$
and therefore, using Cauchy-Schwarz's inequality in $\ell^2(\N^2)$,
\ba|\ll \A(\x), \y\gg &\le& \|\A\|_{\ell^2(\N^2,\B(H))} \|(\|x_j\|\|y_k\|)\|_{\ell^2(\N^2)}\\
&=& \|\A\|_{\ell^2(\N^2,\B(H))} \|\x\| \|\y\| .\ea
Assume now $\A\in \tilde H^2(\T^2, \B(H))$ and apply Cauchy-Schwarz in $L^2(\T^2)$
\ba &&|\int_0^{2\pi} \int_0^{2\pi}\J \A_{N,M}(s,t)(h_\x(s)\otimes h_\y(t))\ds\dt|\\
&\le& \|\A_{N,M}\|_{H^2_0(\T^2,\B(H))}\|h_\x\|_{H^2_0(\T,H)} \|h_\y\|_{H^2_0(\T,H)}.\ea
Now the result  follows from Proposition \ref{l3}
\end{proof}

Actually a sufficient condition better than $\A\in \ell^2(\N^2, \B(H))$ is given in the following result.
\begin{prop} Let $\A=(\Tkj)\subset \B(H)$ such that  $\Cj$ for all $j\in \N$ or $ {\bf R}_k^*\in \lsot$ for all $k\in \N$ and satisfy $$\min\{\|(\Cj)\|_{\ell^2(\N, \lsot)},
  \|({\bf R}_k^*)\|_{\ell^2(\N, \lsot)}\}=M<\infty.$$
  Then $\A\in \Bl2$ and $\|\A\|\le M.$
\end{prop}
\begin{proof} Let $\x,\y\in \l2$, we have
 \ba
 |\ll \A(\x),\y\gg |
 &\le&   \sk \sj \|y_k\| \|\Tkj(\frac{x_j}{\|x_j\|})\|\|x_j\|\\
 &\le&  (\sj \sk \|\Tkj(\frac{x_j}{\|x_j\|})\|^2)^{1/2}(\sk \sj \|y_k\|^2\|x_j\|^2)^{1/2}\\
 &\le& \|\x\|_{\l2}\|\y\|_{\l2}(\sj \|\Cj\|_{\lsot}^2)^{1/2}.
 \ea
 Similar argument works with ${\bf R_k^*}$ which completes the proof.
\end{proof}

Let us now present some necessary conditions for $\A\in \Bl2$.

Since $\ll  \A(x\ej),y\ek\gg  = \la \Tkj(x),y\ra$ we have that if $\A\in \Bl2$ then $\A\in \ell^\infty(\N^2, \B(H))$ and  $\sup_{k,j} \|\Tkj\|\le \|\A\|.$
\begin{lema} \label{l1} Let  $\A=(\Tkj)\in \Bl2$.  Then
$(\Cj)_j,  ({\bf R}_k)_k, (\Cj^*)_j,  ({\bf R}_k^*)_k\in \ell^\infty(\N,\lsot)$.
\end{lema}
\begin{proof}
 Since for each $\y\in \l2$, $x,y\in H$ and $k,j\in \N$ we have
 $$\ll  \A(x\ek),\y\gg =\ll  {\bf R}_k(x),\y\gg  $$
 and
 $$\ll  \A(\x),y\ej\gg =\ll  \x,\Cj(y)\gg  $$
 we clearly have
 $$\|{\bf R}_k\|_{\lsot}=\sup_{\|x\|=1} \sup_{\|\y\|_{\l2}=1}| \ll  \A(x\ek),\y\gg |\le \|\A\|.$$
 A similar argument allows to obtain $\|\Cj\|_{\lsot}\le \|\A\|.$
 Now since $\|\Tkj\|=\|\Tkj^*\|$ applying the fact that rows in $\A^*$ correspond with the adjoint operators in the columns in $\A$ we obtain the other cases.
\end{proof}

Let us give another necessary condition for boundedness to be used later on.
\begin{prop} \label{mf1} Let $\A=(\Tkj)\in \Bl2$. Then
\be
\sk\sj \|\Tkj x_j\|^2\le \|\A\|^2 \sj \|x_j\|^2.
\ee
\end{prop}
\begin{proof} Let $\x\in \l2$ and assume that $\sj \|x_j\|^2=1$. Denote by $F_\x:[0,2\pi]\to \ell^2(H)$ the continuous function given by
$F_\x(s)= \Big( x_j  \fj(s)\Big)$. Trivially we have $\|\x\|=\|F_\x\|_{C(\T, \l2)}$.  Then
\ba
\sk\sj \|\Tkj x_j\|^2&=& \sk \int_0^{2\pi} \|\sj \Tkj x_j\fj(s)\|^2\ds\\
&=& \int_0^{2\pi} \sk \|\sj \Tkj x_j\fj(s)\|^2\ds\\
&=& \int_0^{2\pi} \|\A(F_\x(s))\|^2\ds\\
&\le& \|\A\|^2\int_0^{2\pi} \|F_\x(s)\|^2\ds= \|\A\|^2.
\ea
This concludes the result.
\end{proof}

From Proposition \ref{mf1} we can  get an extension of Schur theorem to  matrices whose entries are operators in $\B(H)$.
\begin{teor} \label{t1} If $\A=(\Tkj)$ and $\bB=(\Skj)$. If $\A,{\bB} \in \Bl2$ then $\A*\bB \in \Bl2$. Moreover
$$\|\A*\bB \|_{\Bl2}\le \|\A \|_{\Bl2}\|\bB \|_{\Bl2}.$$
\end{teor}
\begin{proof}  It suffices to show that if $\x, \y\in c_{00}(H)$ then
\be \label{schur}|\ll  \A*\bB(\x),\y\gg |\le \|\A\|\|\bB\| \|\x\|\|\y\|.\ee
Notice that
\ba
|\ll  \A*\bB(\x),\y\gg | &=& |\sk \la  \sj \Tkj \Skj( x_j),y_k\ra|\\
&=& |\sk \sj \la \Skj( x_j),\Tkj^*(y_k)\ra|\\
&\le & \sk \sj \| \Tkj^*(y_k)\| \| \Skj( x_j)\|\\
&\le & (\sk \sj \| \Tkj^*(y_k)\|^2)^{1/2} (\sk \sj\| \Skj( x_j)\|^2)^{1/2}.
\ea

Using the estimate above, combined with  Proposition \ref{mf1} applied to ${\bf B}$ and $\A^*$, due to the fact $\|\A\|= \|\A^*\|$, one obtains (\ref{schur}). The proof is then complete.
\end{proof}

\bigskip

Given $S\subset \N\times \N$ and $\A=(\Tkj)$ we write
$P_S\A= (\Skj\chi_S)$ that is the matrix with entries $\Tkj$ if $(k,j)\in S$ and $0$ otherwise.  In particular matrices with a single row, column or diagonal correspond to $S=\{k\}\times \N$, $S=\N\times \{j\}$  and $D_l=\{(k,k+l):k\in \N\}$ for $l\in \Z$ respectively. Also the  case of finite  or upper (or lower) triangular matrices coincide with $P_S\A$ for $S=[1,N]\times[1,M]=\{(k,j): 1\le k\le N, 1\le j\le M\}$ or
$S=\Delta=\{(k,j): j\ge k\}$ (or $S=\{(k,j): j\le k\}$ ) respectively.

It is well known that the mapping $\A\to P_S\A$ is not continuous in $\B(H)$ for all sets $S$ (for instance, the reader is referred to \cite[Chap.2, Thm.2.19]{PP} to see that $S=\Delta$ the triangle projection  is unbounded) but there are cases where this holds true. Clearly we have that $\A\in \Bl2$ if and only if $\|\A\|=\sup_{N,M}\|P_{[1,N]\times [1,M]}\A\|<\infty$. This easily follows noticing that
$$\ll  P_{[1,N]\times [1,M]}\A(\x),\y\gg = \ll  \A(P_N\x),P_M\y\gg $$
where $P_N\x$ stands for the projection on the $N$-first coordinates of $\x$,

In general it is rather difficult to compute the norm of the matrix $\A$. Let us point out some trivial cases.
\begin{coro} \label{c1} Let $\A=(\Tkj)\subset \B(H)$. Then

(i) $\|P_{\N\times\{j\}}\A\|= \|\Cj\|_{\lsot}$ for each $j\in \N$.

(ii) $\|P_{\{k\}\times \N}\A\|= \|{\bf R}_k\|_{\lsot}$ for each $k\in \N$.

(iii) $\|P_{D_l}\A\|= \sup_k \|T_{k,k+l}\|$ for each $l\in \Z$ (where $T_{k,k+l}=0$ whenever $k+l\le 0$).
\end{coro}
\begin{proof} (i) and (ii) follow trivially from Lemma \ref{l1}.

To see (iii) note that
$(P_{D_l}\A(\x))_k= (T_{k,k+l}x_{k+l})_k$. Hence $\|P_{D_l}\A(\x)\|\le (\sup_k \|T_{k,k+l}\|)\|\x\|$. Since the other inequality always holds  the proof is complete.
\end{proof}

\section{Toeplitz multipliers on operator-valued matrices}

In this section we shall achieve the operator-valued analogues to the Toeplitz and Bennet theorems presented in the introduction.

\begin{teor}  Let $\A=(\Tkj)\in \mathcal T$. Then $A\in \Bl2$ if and only if there exists $\mu\in V^\infty(\T, \B(H))$ such that $\Tkj=\hat \mu(j-k)$ for all $k,j\in \N$. Moreover $\|\A\|=\|\mu\|_\infty.$
\end{teor}
\begin{proof} Assume that $\mu\in V^\infty(\T, \B(H))$ and $\Tkj=\hat \mu(j-k)$ for all $k,j\in \N$.
Then for $\x,\y\in c_{00}(H)$ we have
\ba
 \ll \A(\x),\y\gg &=&\sum_{k=1}^M\sum_{j=1}^N \la \Tkj(x_j),y_k\ra\\
&=&\sum_{k=1}^M\sum_{j=1}^N \la  T_{\mu}({\overline\fk}\fj)(x_j),y_k\ra\\
&=& \sum_{k=1}^M\sum_{j=1}^N \Psi_{\mu}(\fj x_j\otimes {\overline\fk}  y_k)\\
&=& \Psi_{\mu}(\sum_{k=1}^M\sum_{j=1}^N \fj x_j\otimes {\overline\fk}  y_k)\\
&=& \Psi_{\mu}((\sum_{j=1}^N\fj x_j)\otimes (\sum_{k=1}^M {\overline \fk}y_k))
\ea
Therefore
\ba
|\ll  \A(\x),\y\gg |
&\le &\|\Psi_{\mu}\|_{L^1(\T, H\hat\otimes H)^*}
\int_0^{2\pi}\|h_\x(t)\otimes h_\y(-t))\|_{H\hat\otimes H} \dt\\
&=& \|\mu\|_\infty \int_0^{2\pi}\|h_\x(t) \| \|h_\y(-t)\| \dt\\
&\le& \|\mu\|_{\infty} (\int_0^{2\pi}\|h_\x(t) \|^2 \dt)^{1/2}(\int_0^{2\pi} \|h_\y(t))\|^2 \dt)^{1/2}\\
&\le& \|\mu\|_{\infty}\|\x\|_{\l2}\|\y\|_{\l2}.
\ea
Hence $\A\in \Bl2$ and $\|\A\|\le \|\mu\|_\infty.$

Conversely, let us assume that $\A\in \Bl2$ and $\Tkj= T_{j-k}$ for a given sequence ${\bf T}=(T_n)_{n\in \Z}$ of operators in $\B(H)$.
We define
\be
T(\sum_{n=-M}^N\alpha_n\varphi_n)= \alpha_0 T_{1,1}+\sum_{n=1}^{M}\alpha_{-n} T_{n+1,1}+ \sum_{n=1}^N\alpha_{n} T_{1,n+1}.
\ee
Let us see that  $T\in \L( L^1(\T), \B(H))$. Since $L^1(\T)=\overline{span\{\fk:k\in \Z\}}^{\|\cdot\|_1}$ it suffices to show that
\be \label{eq1} \|T(\sum_{n=-M}^N\alpha_n\varphi_n)\|\le \|\A\| \int_0^{2\pi}|\sum_{n=-M}^N\alpha_n\varphi_n(t)|\dt.\ee
Let $x,y\in H$ and notice that
$$\la T(\sum_{n=-M}^N\alpha_n\varphi_n)(x),y\ra= \sum_{n=-M}^N\alpha_n\beta_{n}(x,y)$$
where $\beta_n(x,y)= \la  T_{n}(x),y\ra$.
Now taking into account that $A_{x,y}=(\la \Tkj(x),y\ra)$ is a Toeplitz matrix and defines a bounded operator $A_{x,y}\in \B(\ell^2)$ with $\|A_{x,y}\|\le \|\A\|\|x\|\|y\|$ we obtain, due to Theorem \ref{tt}, that $$\psi_{x,y}=\sum_{n\in \Z} \beta_n(x,y)\varphi_n\in L^\infty(\T)$$ with  $\|\psi_{x,y}\|_{L^\infty(\T)}\le \|\A\| \|x\|\|y\|.$ Finally we have
\ba |\la T(\sum_{n=-M}^N\alpha_n\varphi_n)(x),y\ra|&=&|
\int_0^{2\pi}(\sum_{n=-M}^N\alpha_n\varphi_n(t))\psi_{x,y}(-t)\dt|\\
&\le& \|\sum_{n=-M}^N\alpha_n\varphi_n(t)\|_{L^1(\T)}\|\A\|\|x\|\|y\|.
\ea
This shows (\ref{eq1}) which gives $\|T\|_{L^1(\T)\to \B(H)}\le \|\A\|.$
Finally, from the embedding $C(\T)\to L^1(\T)$ we have that there exists $\mu\in V^\infty(\T,\B(H))$ such that $T_\mu=T$ and $\|\mu\|_\infty\le \|A\|.$ The proof is then complete.
\end{proof}

To prove the analogue of Bennet't theorem on Schur multipliers we shall need the following lemmas.

\begin{lema} \label{ultimo} Let $\A=(\Tkj)\in \M_l(\l2)\cup \M_r(\l2)$ and $x_0,y_0\in H$ with $\|x_0\|=\|y_0\|=1$.  Denote by $A_{x_0,y_0}=(\gamma_{kj})$ the matrix with entries $$\gamma_{kj}=\la  \Tkj(x_0),y_0\ra, \quad k,k\in \N.$$ Then $A_{x_0,y_0}\in \M(\ell^2)$ and $\|A_{x_0,y_0}\|_{\M(\ell^2)}\le \min\{\|\A\|_{\M_l(\l2)}, \|\A\|_{\M_r(\l2)}\}.$
\end{lema}
\begin{proof} Let $z_0\in H$ and $\|z_0\|=1$ and  consider the bounded operators $\pi_{z_0}:\l2\to \ell^2$ and $i_{z_0}: \ell^2\to \l2$ given by
 $$ \pi_{z_0}((x_j))= (\la x_j,z_0\ra)_j, \quad i_{z_0}((\alpha_k))= (\alpha_k z_0)_k.$$

 Now given $B=(\beta_{kj})\in \B(\ell^2)$ with $\|B\|=1$ we define $\bB= i_{z_0} B \,\pi_{z_0}$.

 Hence $\bB\in \Bl2$.  Moreover $\|\bB\|= \|B\|$ because $\|i_{z_0}\|=\|\pi_{z_0}\|=1$ and $B((\alpha_j)) z_0= \bB( (\alpha_j z_0))$ for any $(\alpha_j)\in \ell^2$.

 Let us write $\bB=(\Skj)$ and observe that $\Skj=\beta_{kj}\widetilde{z_0\otimes z_0}$. Indeed, $$\la\Skj(x), y\ra= \ll  \bB(x\ej),y\ek\gg = \ll  (\la x,z_0\ra\beta_{kj}z_0)_k, y\ek\gg =\beta_{kj}\la x,z_0\ra\la z_0,y\ra. $$

 Recall that $T(\widetilde{x\otimes y})= \widetilde{x\otimes T(y)}$ and
  $(\widetilde{x\otimes y})T= \widetilde{T^*x\otimes y}$ for any $T\in \B(H)$ and $x,y\in H$. In particular we obtain $$\la (\Tkj \Skj)(x_0),y_0\ra= \beta_{kj}\la \Tkj (z_0),y_0\ra \la x_0,z_0\ra$$
  and
  $$\la (\Skj \Tkj)(x_0),y_0\ra= \beta_{kj}\la \Tkj (x_0),z_0\ra \la z_0,y_0\ra $$
Therefore, choosing $z_0=x_0$ and ${\bf C}=\A* \bB$ one has $C_{x_0,y_0}= A_{x_0,y_0}*B$ and using that $\| C_{x_0,y_0}\|\le \|{\bf C}\|$ we obtain $$\|A_{x_0,y_0}*B\|_{\B(\ell^2)}\le \|\A*\bB\|_{\Bl2}\le \|\A\|_{\M_l(\l2)}.$$
Similarly choosing $z_0=y_0$ and ${\bf C}= \bB*\A$ one obtains
$$\|B*A_{x_0,y_0}\|_{\B(\ell^2)}\le\|\A\|_{\M_r(\l2)}.$$
This completes the proof.
\end{proof}

\begin{lema} \label{mainlema} Let $\mu\in {\mathfrak M}(\T, \B(H))$,  $\A=(\Tkj)\in \mathcal T$ with $\Tkj=\hat\mu(j-k)$ for $k,j\in \N$, $\bB=(\Skj)\subset \B(H)$ and $\x,\y\in c_{00}(H)$. Then
\be \label{mult}
\ll  \A*\bB(\x),\y\gg =\Psi_{\mu}(\int_0^{2\pi}\int_0^{2\pi}{\bB}_{N,M}(\cdot-s, \cdot-t)(h_{\x}(s))\otimes h_{\y}(-t) \ds\dt\Big).
\ee
\end{lema}
\begin{proof} Let $\x,\y\in c_{00}(H)$, say $h_\x=\sum_{j=1}^N x_j\fj$ and $h_\y=\sum_{k=1}^M y_k\fk.$ Recall that $x_j=\int_0^{2\pi} h_{\x}(s)\overline{\fj(s)} \ds$ and $y_k=\int_0^{2\pi} h_{\y}(t)\overline{\fk(t)} \dt$ .
Then
\ba
&&\ll  \A*\bB(\x),\y\gg =\\
&=&\sum_{k=1}^M\sum_{j=1}^N \la \hat\mu(j-k)\Skj(x_j),y_k\ra\\
&=& \int_0^{2\pi} \la \sum_{k=1}^M(\sum_{j=1}^N\hat\mu(j-k)\Skj(x_j))\fk(t),h_{\y}(t)\ra\dt\\
&=& \int_0^{2\pi} \la \sum_{l=-M}^N\hat\mu(l)(\sum_{j-k=l}\Skj(x_j)\fk(t)),h_{\y}(t)\ra\dt\\
&=& \int_0^{2\pi} \int_0^{2\pi}\la \sum_{l=-M}^N\hat\mu(l)\Big((\sum_{j-k=l}\Skj\overline{\fj(s)}\fk(t))(h_\x(s))\Big),h_{\y}(t)\ra\ds\dt\\
&=& \int_0^{2\pi} \int_0^{2\pi} \sum_{l=-M}^N\J\mu(l)\Big((\sum_{j-k=l}\Skj\overline{\fj(s)}\fk(t))(h_\x(s))\otimes h_{\y}(t)\Big)\ds\dt\\
&=&  \sum_{l=-M}^N\J\mu(l)\Big(\int_0^{2\pi} \int_0^{2\pi}(\sum_{j-k=l}\Skj\overline{\fj(s)}\fk(t))(h_\x(s))\otimes h_{\y}(t)\ds\dt\Big)\\
&=&  \Psi_{\mu}\Big( \sum_{l=-M}^N \big(\int_0^{2\pi} \int_0^{2\pi} (\sum_{j-k=l}\Skj\overline{\fj(s)}\fk(t))(h_\x(s))\otimes h_{\y}(t)\ds\dt)\varphi_l\Big)\\
&=&  \Psi_{\mu}\Big( \int_0^{2\pi} \int_0^{2\pi}\big(\sum_{k=1}^M \sum_{j=l}^N\Skj\overline{\fj(s)}\fk(t)\varphi_j \varphi_{-k}\big)(h_\x(s))\otimes h_{\y}(t)\ds\dt\Big)\\
&=& \Psi_{\mu}(\int_0^{2\pi}\int_0^{2\pi}{\bB}_{N,M}(s-\cdot, t-\cdot)(h_{\x}(s))\otimes h_{\y}(t) \dt\ds\Big).
\ea
The proof is complete.
\end{proof}
\begin{teor}  If $\mu\in M(\T, \B(H))$ and $\A=(\Tkj)\in \mathcal T$ with $\Tkj=\hat\mu(j-k)$ for $k,j\in \N$ then $\A\in {\mathcal M}_l(\l2)\cap \Ml2$ and
$$\max\{\|\A\|_{{\mathcal M}_l(\l2)}, \|\A\|_{{\mathcal M}_r(\l2)}\}\le |\mu|.$$
\end{teor}
\begin{proof} Since $\|\A\|_{{\mathcal M}_l(\l2)}=\|\A^*\|_{{\mathcal M}_l(\l2)}$ and $|\mu|=|\mu^*|$ then it suffices to show the case of left Schur multipliers.   Let $\x,\y\in c_{00}(H)$ and $\bB=(S_{kj})\subset \B(H)$ such that $\bB\in \Bl2$.
Define $$G(u)= \int_0^{2\pi}\int_0^{2\pi}{\bB}_{N,M}(s-u, t-u)(h_{\x}(s))\otimes h_{\y}(t) \dt\ds.$$
Hence we can rewrite
$$G(u)= \sk\sj\Skj(x_j\fj(u))\otimes y_k\overline{\fk(u)}.$$

In particular \ba
\|G(u)\|_{H\hat\otimes H}&\le& \sk\|\sj\Skj(x_j\fj(u))\| \| y_k\overline{\fk(u)} \|\\
&\le&(\sk\|\sj\Skj(x_j\fj(u))\|^2)^{1/2} \|\y\|\\
&\le&\|{\bf B}\| \|\x\|\|\y\|\ea
From Lemma \ref{mainlema} we have
\ba
|\ll \y, \A* {\bf B}(\x)\gg |&\le &\|\Psi_{\mu}\|_{C(\T, H\hat\otimes H)^*}
\sup_{0\le u<2\pi}\|G(u)\|_{H\hat\otimes H}\\
&=& |\mu|\|\bB\|\|\x\|\|\y\|.
\ea
 This finishes the proof.
\end{proof}

\begin{lema} \label{mainlema2} Let $\mu,\nu \in {\mathfrak M}(\T, \B(H))$,  $\A=(\Tkj)\in \mathcal T$ with $\Tkj=\hat\mu(j-k)$, $\bB=(\Skj)\in \mathcal T$ with $\Skj=\hat\nu(j-k)$ for $k,j\in \N$ and $\x,\y\in c_{00}(H)$. Then
\be \label{mult2}
\ll  \A*\bB(\x),\y\gg =\Psi_{\mu}\Big( \sum_{k=1}^M \big(\sum_{j=1}^N \hat\nu(j-k)( x_j)\fj\big)\otimes y_{k}\bar\fk \Big)
\ee
\end{lema}
\begin{proof}  Denote $h_{\x}=\sum_{k=1}^My_k\fk$ and $h_{\y}=\sum_{j=1}^N x_j\fj$.
Then
\ba
\ll  \A*\bB(\x),\y\gg &=&\sum_{k=1}^M\sum_{j=1}^N \la \hat\mu(j-k)\hat\nu(j-k)(x_j),y_k\ra\\
&=&  \sum_{l=-M}^N \sum_{k=1}^M\la \hat\mu(l)\hat\nu(l)(x_{j+l}),y_{k}\ra\\
&=&  \sum_{l=-M}^N\sum_{k=1}^M \J\hat\mu(l)\Big(\nu(l)(x_{k+l})\otimes y_{k}\Big)\\
&=&  \sum_{l=-M}^N \J\hat\mu(l)\Big(\sum_{k=1}^M \hat\nu(l) (x_{k+l})\otimes y_{k}\Big)\\
&=& \Psi_{\mu}\Big(\sum_{l=-M}^N \Big(\sum_{k=1}^M \hat\nu(l)(x_{k+l})\otimes y_{k}\Big)\varphi_l \Big)\\
&=& \Psi_{\mu}\Big( \sum_{k=1}^M (\sum_{j=1}^N \hat\nu(j-k)(x_j)\fj)\otimes y_{k}\bar\fk \Big).\\
\ea
The proof is complete.
\end{proof}

\begin{coro} \label{maincor} Let  $\A=(\Skj)\in \mathcal T$ such that $\Skj=\hat\nu(j-k)$ for some $\nu\in {\mathfrak M}(ºT, \B(H))$. For  each $\x,\y\in c_{00}(H)$ we denote
$$F_{\x,\y,\A}(t)=\sum_{k=1}^\infty (\sj \hat\nu(j-k)(x_j)\fj(t))\otimes y_{k}\bar\fk(t)$$

If $\A\
 \in \M_r(\l2)$ then
$$\|F_{\x,\y, \A}\|_{L^1(\T,H\hat\otimes H)}\le \|\A\|_{\M_r(\l2)}\|\x\|_{\l2}\|\y\|_{\l2}.$$
\end{coro}
\begin{proof}  If $\A\
 \in \M_r(\l2)$ then $\bB* \A\in \Bl2$ for any $\bB\in \Bl2\cap \mathcal T$. In particular for any $\bB=(\Tkj)$ with $\Tkj=\hat\mu(j-k)$ for some $\mu\in V^\infty(\T,\B(H))$ with $\|\mu\|_\infty= \|\bB\|$.  Since $L^1(\T, H\hat\otimes H)\subseteq (V^\infty(\T, B(H)))^*$ isometrically, we can use Lemma \ref{mainlema2} to obtain
\ba
\|F_{\x,\y, \A}\|_{L^1(\T,H\hat\otimes H)}&=& \sup \{|\Psi_\mu (F_{\x,\y, \A})|: \|\mu\|_\infty=1\}\\
&=& \sup\{ |\ll \bB*\A (\x),\y\gg| : \|\bB\|=1 \}\\
&\le&\|\A\|_{\M_r(\l2)}\|\x\|_{\l2}\|\y\|_{\l2}.
\ea
This completes the proof.
\end{proof}

\begin{teor}  Let $\A=(\Tkj)\in \mathcal T\cap \M_r(\l2)$. Then there exists $\mu\in M_{SOT}(\T, \B(H))$ such that  $\Tkj=\hat\mu(j-k)$ for all $k,j\in \N$.  Moreover $\|\mu\|_{SOT}\le \|\A\|_{\Ml2}.$
\end{teor}
\begin{proof} Let $\A\in \M_r(\l2)$. For each $x_0,y_0\in H$, as above we consider the scalar-valued Toeplitz matrix $A_{x_0,y_0}=(\la  \Tkj(x_0),y_0\ra$. Using Lemma \ref{ultimo} we have that $A_{x_0,y_0}\in \M(\ell^2)$ and $\|A_{x_0,y_0}\|_{\M(\ell^2)}\le \|\A\|_{\M(\l2)}.$
This guarantees invoking Theorem \ref{tb} that there exists $\eta_{x_0,y_0}\in M(\T)$ such that $\la  \Tkj(x_0),y_0\ra= \widehat {\eta_{x_0,y_0}}(j-k)$ for all $j,k\in \N$  and $|\eta_{x_0,y_0}|=\|A_{x_0,y_0}\|_{\M_r(\ell^2)}.$

Now define $\mu(A)\in \B(H)$ given by $$\la \mu(A)(x),y\ra= \eta_{x,y}(A), \quad x,y\in H.$$
 Let us show that
$\mu\in M_{SOT}(\T, \B(H))$ and $\|\mu\|_{SOT}\le \|\A\|_{\M_r(\l2)}.$

First we need to show that $\mu(A)\in \B(H)$ for any $A\in {\mathfrak B}(\T)$.
This follows using that $$\widehat{\eta_{\lambda x+\beta x',y}}(l)= \lambda\widehat{\eta_{x,y}}(l)+ \beta\widehat{\eta_{x',y}}(l), \quad l\in \Z$$ for any $\lambda, \beta\in \C$ and $x,x', y\in H$. This guarantees that  $\eta_{\lambda x+\beta x',y}= \lambda\eta_{x,y}+ \beta\eta_{x',y}$ and hence $\mu(A):H\to H$ is a linear map. The continuity follows from the estimate $|\eta_{x,y}|\le \|\A\|_{\M_r(\l2)} \|x\|\|y\|$.  To show that it is a regular measure, select $\{x_n:n\in \N\}$ dense in $ H$. Hence  for any $S\in \B(H)$ we have $$\|S\|=\sup \{ \la S(x_n),x_m\ra : n,m\in \N\}.$$
Denoting by $\eta_{n,m}=\eta_{x_n,x_m}$ we have that for each $B\in {\mathfrak B}(\T)$, given $(n,m)\in \N\times \N$  and $\e>0$ there exists $K_{n,m}\subset B\subset O_{n,m}$ which are compact and open respectively so that
$$|\eta_{n,m}|(O_{n,m}\setminus K_{n,m})<\e$$
Now selecting $K=\overline{\cup_{n,m} K_{n,m}}$ and $O= (\cap_{n,m} O_{n,m})^\circ$ we conclude that $$\|\mu\|(O\setminus K)<\e.$$
This shows that $\mu\in  {\mathfrak M}(\T, \B(H))$.

Using now that $$\la T_\mu(\phi)(x), y\ra= T_{\eta_{x,y}}(\phi)$$ for each $\phi\in C(\T)$, where $T_{\eta_{x,y}}\in \L(C(\T), \C)$ denotes the operator associated to $\eta_{x,y}\in M(\T)$,
 we  clearly have that  $\Tkj=\hat\mu(j-k)$ for all $j,k\in \N$.

Select  $y_k=y\beta_k$ for some $\beta_k \in \C$ and $\|y\|=1$. From  Corollary \ref{maincor} we obtain that
\ba && \int_0^{2\pi} \|\Big(\sum_{k=1}^M \sum_{j=1}^N\hat\mu(j-k)(x_{j})\beta_k\fj(t)\bar\fk(t)\Big)\otimes y\|_{H\hat\otimes H} \dt\\
&=&\int_0^{2\pi} \|\sum_{l=-M}^N  \hat\mu(l)(\sum_{k=1}^M x_{l+k})\beta_k)\varphi_l(t)\|\dt\\
&\le& \|\A\|_{\M_r(\l2)}\|\x\|_{\l2}(\sum_{k=1}^M|\beta_k|^2)^{1/2}\ea
Now select $x_j=x \alpha_j$ for $\|x\|=1$ to get
$$\int_0^{2\pi} \|\sum_{l=-M}^N  \hat\mu(l)(x)(\sum_{j-k=l}\alpha_j\fj(t)\beta_k\bar\fk(t))\|\dt\le \|\A\|_{\M_r(\l2)}(\sum_{j=1}^N|\alpha_j|^2)^{1/2}(\sum_{j=1}^N|\beta_j|^2)^{1/2}. $$
Using now
$$\gamma(s)= \sum_{l=-M}^N (\sum_{j-k=l}\beta_k\alpha_j)\varphi_l(s)$$
Now recall that $\hat\mu(l)(x)=\hat\mu_x(l)$ and
$$\sum_{l=-M}^N  \hat\mu_x(l)(\sum_{j-k=l}\alpha_j\fj(t)\beta_k\bar\fk(t))=\int_0^{2\pi}
\sum_{l=-M}^N  \Big(\hat\mu_x(l)\varphi_l(s)\Big) \gamma(t-s) \ds.$$
Therefore, if $\alpha=\sum_{j=1}^\infty \alpha_j \fj$ and $\beta=\sum_{k=1}^\infty\beta_k\fk$ belong to $L^2(\T)$, we have that
$\gamma(t)=\alpha(t)\beta(-t)$ and
\be \label{ultima}
\int_0^{2\pi} \|\mu_x*\gamma(t)\|\dt\le \|\A\|_{\Ml2}\|\alpha\|_{L^2(\T)}\|\beta\|_{L^2(\T)}.
\ee

To show that $\mu_x\in M(\T,H)$, due to Lemma \ref{convo}, it suffices to prove that
\be \label{fin}\sup_{0<r<1}\|\mu_x * P_r\|_{L^1(\T, H)}<\infty.\ee

Choosing $\beta(t)=\alpha(t)=\frac{\sqrt{1-r^2}}{|1-re^{it}|}$ we obtain that $\gamma(t)=P_r(t)$ and from (\ref{ultima}) we get (\ref{fin})and the estimate $\|\mu_x\|_{M(\T,H)}\le \|\A\|_{\M_r(\l2)}.$
This finishes the proof.
\end{proof}

\vspace{.1in}
\[
\begin{tabular}{lccl}
Departamento de An\'alisis Matem\'{a}tico &   \\
Universidad de Valencia &  \\
46100 Burjassot &   \\
Valencia &   \\
Spain &  &  & \\
oscar.blasco@uv.es & Ismael.Garcia-Bayona@uv.es \\
\end{tabular}
\]


\begin{thebibliography}{CoifRoch}

\bibitem{AP} Alexandrov, A. B. ; Peller, V.V. \textit{Hankel and Toeplitz-Schur multipliers}, Math.
Ann., 324 (2002), 277--327.


\bibitem{Be} Bennet, G. \textit{Schur multipliers} Duke Math. J. 44 (1977), 603–-639.

\bibitem{B} Blasco, O. \textit{Fourier Analysis on vector measures on locally compact abelian groups}
Revista de
la Real Academia de Ciencias Exactas, F{\'i}sicas y Naturales. Serie A. Matem{\'a}ticas,
2016, 110, 2, 519--539.

\bibitem{BG} B\"ottcher, A; Grudsky, S. \textit{Toeplitz Matrices, Asymptotic Linear
Algebra, and Functional Analysis.} Hindustan Book Agency, New
Delhi, 2000 and Birkh\"auser Verlag, Basel,Boston, Berlin, 2000.



\bibitem{DFS} Diestel, J.; Fourie, J.; Swart, J. \textit{The metric theory of tensor products. Grothendieck's R\'esum\'e Revisited} , American Math. Soc.2008.


\bibitem{DU}  Diestel,J.;  Uhl, J.J. \textit{Vector Measures}, Math. Surveys vol 15 Amer. Math. Soc. Providence (1977)

\bibitem{DJT}  Diestel,J.;  Jarchow, H., Tonge, A. \textit{Absolutely
summing operators}, Cambridge University Press, (1995).

\bibitem{D}  Dinculeanu, N. \textit{Vector Measures}, VEB Deutscher Verlag der Wissenschaften. Berlin (1966)

\bibitem{H} Hengsen, W. \textit{A simpler proof of Singer's representation theorem,} Proc. Amer. Math. Soc.
124 (1996), 3211-3212.

    \bibitem{HNVW} Hytonen, T.;
 van Neerven, J.;
Veraar, M.;
Weis, L.  \textit{
Analysis in Banach spaces. Vol. I. Martingales and Littlewood-Paley theory.}
Series of Modern Surveys in Mathematics, 63.
Springer, Cham, 2016. xvi+614 pp.



\bibitem{PP}Persson, L-E.; Popa, N.
 \textit{Matrix Spaces and Schur Multipliers: Matriceal Harmonic Analysis}, NJ : World Scientific, 2014. 208 s.

\bibitem{Ryan} Ryan, R.A, \textit{Introduction to tensor products of Banach Spaces}, Springer Monographs in Mathematics. Springer-Verlag London Ltd., London, 2002.







\bibitem{Schur} Schur, J. \textit{Bemerkungen zur Theorie der beschr\"ankten Bilinearformen mit unendlich vielen
Verandlichen.} J. Reine Angew. Math. 140 (1911), 1–28.

\bibitem{S} Singer, I. \textit{Linear functionals on the space of continuous mappings of a compact Hausdorf
space into a Banach space} (in Russian), Rev. Roum. Math. Pures Appl. 2 (1957), 301--315.MR 20:3445


\bibitem{S2} Singer, I. \textit{Sur les applications lin\'eaires int\'egrales des espaces de fonctions continues. I,} Rev.
Roum. Math. Pures Appl. 4 (1959), 391--401. MR 22:5883

\bibitem{T} Toeplitz, O. \textit{Zur Theorie der quadratischen und bilinearen Formen
von unendlichvielen Veranderlichen.} Math. Annalen 70
(1911), 351-376.

\end{thebibliography}
\end{document}